\documentclass{amsproc}
\usepackage{bbm}
\usepackage{amsfonts}
\usepackage{mathrsfs} 

\usepackage{amsmath,amssymb}
\newtheorem{theorem}{Theorem}[section]

\newtheorem{prop}[theorem]{Proposition}
\newtheorem{theo}[theorem]{Theorem}
\theoremstyle{definition}
\newtheorem{defi}[theorem]{Definition}

\newtheorem{coro}[theorem]{Corollary}
\theoremstyle{remark}
\newtheorem{remark}[theorem]{Remark}

\numberwithin{equation}{section}

\def\cal{\mathcal}

\def\om{\omega}

\def\om{\omega}

\newfont{\df}{eufm10}
\def\vep{\varepsilon}

\def\ot{\otimes}

\def\ot{\otimes}

\def\ra{\rangle}
\def\la{\langle}

\def\vn{\varepsilon}
\def\om{\omega}
\def\ot{\otimes}
\def\ra{\rangle}
\def\la{\langle}

\begin{document}

\title[On finite-dimensional representations of $U_{r,s}(\widehat{\mathfrak{g}})$]
    {On finite-dimensional representations of two-parameter quantum affine algebras  
    }

\author[Jing]{Naihuan Jing}
\address{Department of Mathematics,
   North Carolina State University,
   Ra\-leigh, NC 27695-8205, USA}
\email{jing@math.ncsu.edu}

\author[Zhang]{Honglian Zhang$^\star$}
\address{Department of Mathematics, Shanghai University,
Shanghai 200444, China} \email{hlzhangmath@shu.edu.cn}

\thanks{$^\star$ H.Zhang, Corresponding Author}

\subjclass{Primary 17B37, 81R50; Secondary 17B35}

\keywords{Two-parameter quantum affine algebra,  Drinfeld
realization, evaluation representation, finite-dimensional
representation. }
\begin{abstract}
We introduce the Drinfeld polynomial for each weight of
two-parameter quantum affine algebras and establish a one-to-one
correspondence between finite irreducible representations and sets
of $l$-tuples of pairs of polynomials with certain conditions.
\end{abstract}

\maketitle

\section{ Introduction}
\smallskip

Quantum groups were introduced independently by V. G. Drinfeld and M. Jimbo
in 1985 using Chevalley generators and Serre relations. A few years later, Drinfeld
\cite{D} found the second definition of quantum affine algebras and
Yangians in terms of root vectors and introduced certain polynomials
to characterize finite dimensional irreducible representations of Yangians. He showed that
the Yangian representations are finite dimensional if and only
if their Drinfeld polynomials are of certain form. Later Chari and
Pressley \cite{CP1}-\cite{CP3} generalized the notion of Drinfeld
polynomials to quantum affine algebras and proved similar results
for finite dimensional representations of both untwisted and twisted
quantum enveloping algebras. Drinfeld polynomials are also related to
Frenkel-Rechetikhin characters \cite{FR} of finite dimensional irreducible representations.

Two-parameter quantum enveloping algebras are generalization of (one-para\-meter)
quantum enveloping algebras, originally introduced as generalization of Hopf algebras \cite{T, J, Do}
and have close connections with Yang-Baxter equations \cite{R}. They were first defined for finite types
using specific forms of Cartan matrices \cite{BW1, BW2} \cite{BGH1, BGH2} \cite{BH, HS}, and later
the affine cases were introduced by \cite{HRZ} and the toroidal cases by the authors
\cite{JZ1} in the context of McKay correspondence. As in the one-parameter cases
various combinatorial realizations for
2-parameter quantum affine algebras were given in \cite{JZ1, JZ2}, where the meaning of the second parameter
was explained.

Motivated by Drinfeld's and Chari-Pressley's work, we study finite dimensional representations of
two-parameter quantum affine algebras in this paper. We introduce
Drinfeld polynomials for each finite dimensional irreducible representations and show
that they are again characterized by Drinfeld polynomials. Although the theory is much expected,
there are some new features in
two-parameter situation. The most notable one is that there are in
fact a pair of related two-parameter Drinfeld polynomials for each
finite dimensional representation instead of one for the positive
root vector and negative root vector. For completeness,
we provide all proofs in
the case of $\hat{\mathfrak{sl}_2}$, in particular we elaborate more for
the cases when two-parameter quantum enveloping algebras have
distinct features. Just as in the usual case, the two-parameter cases are essentially
in one-to-one correspondence with the one-parameter quantum groups.

This article is organized as follows. After a quick introduction of
two-parameter quantum affine algebras in Section 2, we define the
category of finite dimensional representations of 2-parameter
quantum enveloping algebras in Section 3 and introduce the notion of
Drinfeld polynomials. A general result is shown as in the
one-parameter case. In Section 4 we study the
case of $U_{r, s}(\widehat{\mathfrak{sl}}_2)$ in details and prove the existence of
a pair of Drinfeld polynomials for each finite dimensional representation.
Finally in Section 5 we discuss some specializations of two parameters.

\smallskip
\section{Two-parameter quantum affine algebras}

\noindent {\bf 2.1} \, In this subsection, we recall the definition
of two-parameter quantum
 algebras $U_{r,s}(\widehat{\mathfrak{g}})$, developed in \cite{HRZ},
 \cite{HZ1} and \cite{HZ2}. Here we let $\widehat{\mathfrak{g}}$ be an
 untwisted affine Lie algebra.

Let $\mathbb{K}=\mathbb{Q}(r,s)$ be a field of rational functions
with two indeterminates $r , s$. Let $\Phi$ be the finite root system
of $\mathfrak{g}$ with $\Pi$, a base of simple roots, which is a
subset of a Euclidean space ${\mathbb{R}}^n$ with an inner product
$(\, , \,)$. Let $\epsilon_1,\epsilon_2,\cdots,\epsilon_n$ denote an
orthonormal basis of ${\mathbb{R}}^n$. Let $\delta$ denote the
primitive imaginary root  of the affine Lie algebra $\hat{\mathfrak
g}$, and let $\theta$ be the highest root of the simple Lie algebra
${\mathfrak g}$. Define $\alpha_0=\delta-\theta$ , ~then
$\Pi'=\{\alpha_i\mid i\in I_0\}$ is a basis of simple roots of the
affine Lie algebra $\hat{\mathfrak g}$

We recall that the quantum number in two parameters is defined
by
\begin{equation*}
[n]=\frac{r^n-s^n}{r-s}=r^{n-1}+r^{n-2}s+\cdots+rs^{n-2}+s^{n-1},
\end{equation*}
and the Guassian numbers are defined similarly.

\begin{defi}(\cite{HRZ},\,\cite{HZ1},\,\cite{HZ2})\,   {\bf Two-parameter quantum affine algebra
$U_{r,s}(\widehat{\mathfrak{g}})$}  is the unital associative algebra
over $\mathbb{K}$ generated by the elements $e_j,\, f_j,\,
\omega_j^{\pm 1},\, \omega_j'^{\,\pm 1}\, (j\in I_0=\{0,\,\cdots ,
n\}),\, \gamma^{\pm\frac{1}{2}},\, {\gamma'}^{\pm\frac{1}{2}}$, satisfying the following relations:

\noindent $(\hat{R}1)$ $\gamma^{\pm\frac{1}2},\,\gamma'^{\pm\frac{1}2}$ are central with
 $\gamma\gamma'=(rs)^c$ such
that $\omega_i\,\omega_i^{-1}=\omega_i'\,\omega_i'^{\,-1}=1$, and
\begin{equation*}
\begin{split}[\,\omega_i^{\pm 1},\omega_j^{\,\pm 1}\,]=[\,\omega_i^{\pm 1},\omega_j'^{\,\pm 1}\,]=[\,\omega_i'^{\pm
1},\omega_j'^{\,\pm 1}\,]=0.
\end{split}
\end{equation*}
$(\hat{R}2)$ \ For $\,i,\, j\in I_0$,
\begin{equation*}
\begin{array}{ll}
&\omega_j\,e_i\,\omega_j^{\,-1}=\langle i, j\rangle
\,e_i,\qquad\quad \omega_j\,f_i\,\omega_j^{\,-1}=\langle i, j\rangle^{-1}\,f_i,\\
&\omega'_j\,e_i\,\omega'^{\,-1}_j=\langle j, i\rangle^{-1} \,e_i,
\qquad\ \ \omega'_j\,f_i\,\omega'^{\,-1}_j=\langle j, i\rangle\,f_i,
\end{array}
\end{equation*}\\
$(\hat{R}3)$ \ For $\,i,\, j\in I_0$, we have
 $$[\,e_i, f_j\,]=\frac{\delta_{ij}}{r-s}(\omega_i-\omega'_i).$$
$(\hat{R}4)$  For any $i\ne j$, we have the $(r,s)$-Serre relations:
\begin{gather*}
\bigl(\text{ad}_l\,e_i\bigr)^{1-a_{ij}}\,(e_j)=0,\\
\bigl(\text{ad}_r\,f_i\bigr)^{1-a_{ij}}\,(f_j)=0,
\end{gather*}
where the definitions of the left-adjoint action $\text{ad}_l\,e_i$
and the right-adjoint action $\text{ad}_r\,f_i$ are given in the
following sense:
$$
\text{ad}_{ l}\,a\,(b)=\sum_{(a)}a_{(1)}\,b\,S(a_{(2)}), \quad
\text{ad}_{ r}\,a\,(b)=\sum_{(a)}S(a_{(1)})\,b\,a_{(2)}, \quad
\forall\; a, b\in U_{r,s}(\hat{\mathfrak g}),
$$
where $\Delta(a)=\sum_{(a)}a_{(1)}\ot a_{(2)}$ is given by
proposition \ref{hopf} below,
\smallskip
and the structural constant $\langle i,\, j\rangle$ is the $(i,
j)$-entry of the two-parameter quantum Cartan matrix $J$, given as
follows respectively.

For type $\mathrm{A}_{n}^{(1)}$ with $n>1$,

$$J=\left(\begin{array}{cccccc}
rs^{-1}& r^{-1}& 1 & \cdots & 1 & s \\
s & rs^{-1} & r^{-1}  & \cdots & 1 & 1\\
\cdots &\cdots &\cdots & \cdots & \cdots & \cdots\\
1 & 1 & 1  & \cdots & rs^{-1} & r^{-1}\\
 r^{-1} & 1 & 1 & \cdots & s & rs^{-1}
\end{array}\right)$$

 and  $n=1$,
$$J=\left(\begin{array}{cc}
rs^{-1}& r^{-1}s \\
r^{-1}s & rs^{-1}
\end{array}\right)$$

For type $\mathrm{B}_{n}^{(1)}$.
$$J=\left(\begin{array}{cccccc}
rs^{-1}& (rs)^{-1}& r^{-1} & \cdots & 1 & rs \\
rs & rs^{-1} & r^{-1}  & \cdots & 1 & 1\\
\cdots &\cdots &\cdots & \cdots & \cdots & \cdots\\
1 & 1 & 1  & \cdots & rs^{-1} & r^{-1}\\
 (rs)^{-1} & 1 & 1 & \cdots & s & r^{\frac{1}{2}}s^{-\frac{1}{2}}
\end{array}\right)$$

For type $\mathrm{C}_{n}^{(1)}$.
$$J=\left(\begin{array}{cccccc}
rs^{-1}& r^{-1}& 1 & \cdots & 1 & rs \\
s & r^{\frac{1}{2}}s^{-\frac{1}{2}} & r^{-\frac{1}{2}}  & \cdots & 1 & 1\\
\cdots &\cdots &\cdots & \cdots & \cdots & \cdots\\
1 & 1 & 1  & \cdots & r^{\frac{1}{2}}s^{-\frac{1}{2}} & r^{-1}\\
 (rs)^{-1} & 1 & 1 & \cdots & s & rs^{-1}
\end{array}\right)$$
\end{defi}

For type $\mathrm{D}_{n}^{(1)}$.
$$J=\left(\begin{array}{cccccc}
rs^{-1}& (rs)^{-1}& r^{-1} & \cdots & 1 & (rs)^2 \\
rs & rs^{-1} & r^{-1}  & \cdots & 1 & 1\\
\cdots &\cdots &\cdots & \cdots & \cdots & \cdots\\
1 & 1 & 1  & \cdots & rs^{-1} & (rs)^{-1}\\
 (rs)^{-2} & 1 & 1 & \cdots & rs & rs^{-1}
\end{array}\right)$$

For type  $\mathrm{E}_{6}^{(1)}$ $$J=\left(\begin{array}{ccccccc}
rs^{-1}& (rs)^{-1}& r^{-2}s^{-1} &(rs)^{-1} & rs & rs & rs\\
rs & rs^{-1} & 1 & r^{-1} & 1 & 1 & 1\\
rs^{2} &1 &rs^{-1} & 1 & r^{-1} &1 &1 \\
rs & s & 1 & rs^{-1} & r^{-1} & 1 & 1\\
(rs)^{-1}& 1 & s & s& rs^{-1}& r^{-1}& 1\\
(rs)^{-1}& 1 &1 & 1& s &rs^{-1} & r^{-1} \\
 (rs)^{-1} & 1 & 1 &1 & 1 & s &rs^{-1}
\end{array}\right)$$

For type  $\mathrm{F}_{4}^{(1)}$,$$J=\left(\begin{array}{ccccc}
rs^{-1}& r^{-2}s^{-1}& (rs)^{-1} &rs &  rs\\
rs^2 & rs^{-1} & r^{-1}  & 1 & 1\\
rs &s &rs^{-1} &r^{-1} &1 \\
(rs)^{-1} & 1 & s  & r^{\frac{1}{2}}s^{-\frac{1}{2}} & r^{-\frac{1}{2}}\\
 (rs)^{-1} & 1 & 1 & s^{\frac{1}{2}} & r^{\frac{1}{2}}s^{-\frac{1}{2}}
\end{array}\right)$$

For type $\mathrm{G}_{2}^{(1)}$,
$$J=\left(\begin{array}{cccccc}
rs^{-1}& r^{-2}s^{-1}& rs \\
rs^2 & rs^{-1} & r^{-1}  \\
(rs)^{-1} & s & r^{\frac{1}{3}}s^{-\frac{1}{3}}
\end{array}\right)$$

\medskip


\begin{remark}\, In \cite{JZ1, JZ2} the Fock space realization of two-parametr
quantum affine algebras of types $A$ and $C$ were given in terms of
Young tableaux. The combinatorial model shows that the parameters $r$ and $s$ correspond
roughly the scalars in the insertion and removing operators.
The following fact is straightforward.
\end{remark}

\begin{prop}(\cite{HRZ},\,\cite{HZ1},\,\cite{HZ2})\, \label{hopf}  Two-parameter quantum affine algebra $U=U_{r,s}(\hat{\mathfrak g})$ is a Hopf algebra with
the coproduct $\Delta$, the counit $\vep$ and the antipode $S$
defined below: for $i\in I_0$, we have
\begin{gather*}
\Delta(\om_i^{\pm1})=\om_i^{\pm1}\ot\om_i^{\pm1}, \qquad
\Delta({\om_i'}^{\pm1})={\om_i'}^{\pm1}\ot{\om_i'}^{\pm1},\\
\Delta(e_i)=e_i\ot 1+\om_i\ot e_i, \qquad \Delta(f_i)=1\ot
f_i+f_i\ot \om_i',\\
\Delta(\gamma^{\,\pm\frac{1}2})=\gamma^{\,\pm\frac{1}2}\otimes
\gamma^{\,\pm\frac{1}2},\qquad \Delta(\gamma'^{\,\pm\frac{1}2})=\gamma'^{\,\pm\frac{1}2}\otimes
\gamma'^{\,\pm\frac{1}2},\\
\vn(\om_i^{\pm})=\vn({\om_i'}^{\pm1})=1, \qquad
\vn(e_i)=\vn(f_i)=0,\\
\varepsilon(\gamma^{\pm\frac{1}2})
=\varepsilon(\gamma'^{\,\pm\frac{1}2})=1,\\
S(\gamma^{\pm\frac{1}2})=\gamma^{\mp\frac{1}2},\quad
S(\gamma'^{\pm\frac{1}2})=\gamma'^{\mp\frac{1}2},\\
S(\om_i^{\pm1})=\om_i^{\mp1}, \qquad
S({\om_i'}^{\pm1})={\om_i'}^{\mp1},\\
S(e_i)=-\om_i^{-1}e_i,\qquad S(f_i)=-f_i\,{\om_i'}^{-1}.
\end{gather*}
\end{prop}
\medskip

\noindent {\bf 2.2} \, As is well-known, quantum affine algebras
have another realization called Drinfeld realization \cite{D}.
For two-parameter cases, there also exists
analogous realization isomorphic to the algebra structure defined in
definition 2.1 (\cite{HRZ, HZ1, HZ2}).

\begin{defi}(\cite{HRZ},\,\cite{HZ1},\,\cite{HZ2})\,  The unital
associative algebra ${\mathcal U}_{r,s}(\widehat{\mathfrak {g}})$
 over $\mathbb{K}$  is generated by the
elements  $x_i^{\pm}(k)$, $a_i(\ell)$, $\om_i^{\pm1}$,
${\om'_i}^{\pm1}$, $\gamma^{\pm\frac{1}{2}}$,
${\gamma'}^{\,\pm\frac{1}2}$,  $(i\in I=\{1,\,2,\, \cdots, n\}$,
$k,\,k' \in \mathbb{Z}$, $\ell,\,\ell' \in \mathbb{Z}\backslash
\{0\})$, subject to the following defining relations:

\noindent $(\textrm{D1})$ \  $\gamma^{\pm\frac{1}{2}}$,
$\gamma'^{\,\pm\frac{1}{2}}$ are central such that
$\gamma\gamma'=(rs)^c $,\,
$\omega_i\,\omega_i^{-1}=\omega_i'\,\omega_i'^{\,-1}=1$, and for
\qquad $i,\,j\in I$, one has
\begin{equation*}
\begin{split}
[\,\omega_i^{\pm 1},\omega_j^{\,\pm 1}\,]=[\,\omega_i^{\pm
1},\omega_j'^{\,\pm 1}\,]=[\,\omega_i'^{\pm 1},\omega_j'^{\,\pm
1}\,]=0.
\end{split}
\end{equation*}
$$[\,a_i(\ell),a_j(\ell')\,]
=\delta_{\ell+\ell',0}\frac{
(rs)^{\frac{|\ell|}{2}}(r_is_i)^{-\frac{\ell a_{ij}}2}[\,\ell
a_{ij}\,]_i}{|\ell|}
\cdot\frac{\gamma^{|\ell|}-\gamma'^{|\ell|}}{r-s},
 \leqno(\textrm{D2})
$$
$$[\,a_i(\ell),~\om_j^{{\pm }1}\,]=[\,\,a_i(\ell),~{\om'}_j^{\pm
1}\,]=0.\leqno(\textrm{D3})
$$

$$
\om_i\,x_j^{\pm}(k)\, \om_i^{-1} =  \langle \omega_j',
\omega_i\rangle^{\pm 1} x_j^{\pm}(k), \qquad \om'_i\,x_j^{\pm}(k)\,
\om_i'^{\,-1} =  \langle \omega'_i, \omega_j\rangle
^{\mp1}x_j^{\pm}(k).\leqno(\textrm{D4})
$$
$$
\begin{array}{lll}
[\,a_i(\ell),x_j^{\pm}(k)\,]=\pm\frac{ (rs)^{\frac{|\ell|}{2}} (\la
i,\,i\ra^{\frac{\ell a_{ij}}{2}}- \la i,\,i\ra^{\frac{-\ell
a_{ij}}{2}})}
{\ell(r_i-s_i)}\gamma'^{\pm\frac{\ell}2}x_j^{\pm}(\ell{+}k),\quad
\textit{for} \quad \ell>0,
\end{array}\leqno{(\textrm{D$5_1$})}
$$
$$
\begin{array}{lll}
[\,a_i(\ell),x_j^{\pm}(k)\,]=\pm\frac{ (rs)^{\frac{|\ell|}{2}}(\la
i,\,i\ra^{\frac{\ell a_{ij}}{2}}- \la i,\,i\ra^{\frac{-\ell
a_{ij}}{2}})}
{\ell(r_i-s_i)}\gamma^{\pm\frac{\ell}2}x_j^{\pm}(\ell{+}k), \qquad
\textit{for} \quad \ell<0,
\end{array}\leqno{(\textrm{D$5_2$})}
$$
$$
\begin{array}{lll}
x_i^{\pm}(k{+}1)\,x_j^{\pm}(k') - \langle j,i\rangle^{\pm1} x_j^{\pm}(k')\,x_i^{\pm}(k{+}1)\\
=-\Bigl(\langle j,i\rangle\langle
i,j\rangle^{-1}\Bigr)^{\pm\frac1{2}}\,\Bigl(x_j^{\pm}(k'{+}1)\,x_i^{\pm}(k)-\langle
i,j\rangle^{\pm1} x_i^{\pm}(k)\,x_j^{\pm}(k'{+}1)\Bigr).
\end{array}\leqno{(\textrm{D6})}
$$

$$
[\,x_i^{+}(k),~x_j^-(k')\,]=\frac{\delta_{ij}}{r_i-s_i}\Big(\gamma'^{-k}\,{\gamma}^{-\frac{k+k'}{2}}\,
\om_i(k{+}k')-\gamma^{k'}\,\gamma'^{\frac{k+k'}{2}}\,\om'_i(k{+}k')\Big),\leqno(\textrm{D7})
$$
where $\om_i(m)$, $\om'_i(-m)~(m\in \mathbb{Z}_{\geq 0})$ such that
$\om_i(0)=\om_i$ and  $\om'_i(0)=\om_i'$ are defined as below:
\begin{gather*}\sum\limits_{m=0}^{\infty}\om_i(m) z^{-m}=\om_i \exp \Big(
(r_i{-}s_i)\sum\limits_{\ell=1}^{\infty}
 a_i(\ell)z^{-\ell}\Big),\quad \bigl(\om_i(-m)=0, \ \forall\;m>0\bigr); \\
\sum\limits_{m=0}^{\infty}\om'_i(-m) z^{m}=\om'_i \exp
\Big({-}(r_i{-}s_i)
\sum\limits_{\ell=1}^{\infty}a_i(-\ell)z^{\ell}\Big), \quad
\bigl(\om'_i(m)=0, \ \forall\;m>0\bigr).
\end{gather*}
$$x_i^{\pm}(m)x_j^{\pm}(k)=\langle j,i\rangle^{\pm1}x_j^{\pm}(k)x_i^{\pm}(m),
\qquad\ \hbox{for} \quad a_{ij}=0,\leqno(\textrm{D$8_1$})$$
$$
\begin{array}{lll}
& Sym_{m_1,\cdots
m_{n}}\sum_{k=0}^{n=1-a_{ij}}(-1)^k(r_is_i)^{\pm\frac{k(k-1)}{2}}
\Big[{1-a_{ij}\atop  k}\Big]_{\pm{i}}x_i^{\pm}(m_1)\cdots x_i^{\pm}(m_k) x_j^{\pm}(\ell)\\
&\hskip1.8cm \times x_i^{\pm}(m_{k+1})\cdots x_i^{\pm}(m_{n})=0,
\quad\hbox{for} \quad a_{ij}\neq 0, \quad  1\leq j<i<n,
\end{array} \leqno{(\textrm{D$8_2$})}
$$
$$
\begin{array}{lll}
& Sym_{m_1,\cdots
m_{n}}\sum_{k=0}^{n=1-a_{ij}}(-1)^k(r_is_i)^{\mp\frac{k(k-1)}{2}}
\Big[{1-a_{ij}\atop  k}\Big]_{\mp{i}}x_i^{\pm}(m_1)\cdots x_i^{\pm}(m_k) x_j^{\pm}(\ell)\\
&\hskip1.8cm \times x_i^{\pm}(m_{k+1})\cdots x_i^{\pm}(m_{n})=0,
\quad\hbox{For} \quad a_{ij}\neq 0, \quad  1\leq i<j<n.
\end{array} \leqno{(\textrm{D$8_3$})}
$$
where $[l ]_{\pm{i}},\,[l]_{\mp{i}}$ are defined the same as before
by replacing $r, s$ by $r_i, s_i$ respectively,
$\textit{Sym}_{m_1,\cdots, m_n}$ denotes symmetrization w.r.t. the
indices $(m_1, \cdots, m_n)$.
\end{defi}

Similar to the classical case the Poincare- Birkhoff-Witt theorem also holds \cite{HRZ}.
\smallskip

Let $\mathcal U_{r,s}(\mathfrak{n})$ denote the subalgebra of $\mathcal
U_{r,s}(\widetilde{\mathfrak n})$, generated by $x_i^+(0)$ ($i\in I$).
By definition, it is clear that $\mathcal U_{r,s}(\mathfrak{n})\cong
U_{r,s}(\mathfrak n)$, the subalgebra of $U_{r,s}(\mathfrak{g})$ generated
by $e_i$ ($i\in I$).

The algebra $\mathcal{U}_{r,s}(\widehat{\mathfrak{g}})$ has a triangular
decomposition:
$$\mathcal{U}_{r,s}(\widehat{\mathfrak{g}})=
\mathcal{U}_{r,s}(\widetilde{\mathfrak{n}}^-)\otimes\mathcal{U}_{r,s}^0(\widehat{\mathfrak{g}})
\otimes\mathcal{U}_{r,s}(\widetilde{\mathfrak{n}}),$$ where
$\mathcal{U}_{r,s}(\widetilde{\mathfrak{n}}^\pm)=\bigoplus_{\alpha\in\dot
Q^\pm}\mathcal{U}_{r,s}(\widetilde{\mathfrak{n}}^\pm)_\alpha$ is
generated respectively by $x_i^\pm(k)$ ($i\in I$), and
$\mathcal{U}_{r,s}^0(\widehat{\mathfrak{g}})$ is the subalgebra
generated by $\om_i^{\pm1}$, $\om_i'^{\pm1}$,
$\gamma^{\pm\frac1{2}}$, $\gamma'^{\pm\frac1{2}}$, $D^{\pm1}$,
$D'^{\pm1}$ and $a_i(\pm\ell)$ for $i\in I$, $\ell\in \mathbb{N}$.
Namely, $\mathcal{U}_{r,s}^0(\widehat{\mathfrak{g}})$ is generated by
the toral subalgebra $\mathcal{U}_{r,s}(\widehat{\mathfrak{g}})^0$ and
the quantum Heisenberg subalgebra $\mathcal
H_{r,s}(\widehat{\mathfrak{g}})$ generated by the quantum imaginary
root vectors $a_i(\pm\ell)$ ($i\in I$, $\ell\in \mathbb{N}$).
\medskip

\noindent {\bf 2.3}\, In this subsection, we give some
automorphisms of the two-parameter quantum affine algebras $U_{r,
s}(\widehat{\mathfrak{g}})$. Their proofs are direct verification.

\begin{prop}\, For $I-$tuples $\sigma=(\sigma_0, \cdots,\, \sigma_n)\in \{\pm 1
\}^{n+1}$, there exists an unique automorphism $a_\sigma$ of
two-parameter quantum affine algebra $U_{r, s}(\widehat{\mathfrak{g}})$
such that
$$a_\sigma(\omega_i)=\sigma_i\,\omega_i,\qquad a_\sigma(\omega'_i)=\sigma_i\,\omega'_i,$$
$$a_\sigma(e_i)=\sigma_i\,e_i,\qquad a_\sigma(f_i)=f_i.$$
\end{prop}
\medskip

\begin{prop}\, There exists an automorphism $\Gamma_1$ of
two-parameter quantum affine algebra $\cal{U}_{r,
s}(\widehat{\mathfrak{g}})$ such that
$$\Gamma_1(\gamma^{\pm\frac{1}{2}})=-\gamma^{\pm\frac{1}{2}},\qquad \Gamma_1({\gamma'}^{\pm\frac{1}{2}})=-{\gamma'}^{\pm\frac{1}{2}},
\qquad \Gamma_1(x_i^{\pm}(k))=(-1)^{k}x_i^{\pm}(k)$$
$$\Gamma_1(a_i(k))=a_i(k),\qquad \Gamma_1({\omega}_i)=\omega_i,\qquad \qquad \Gamma_1({\omega'}_i)=\omega'_i.$$
\end{prop}
\medskip

\begin{prop}\, For $a\in \mathbb{C}*$, there exists an unique automorphism $\Gamma_2$ of
two-parameter quantum affine algebra $\cal{U}_{r,
s}(\widehat{\mathfrak{g}})$ such that
$$\Gamma_2(\gamma^{\pm\frac{1}{2}})=\gamma^{\pm\frac{1}{2}},\qquad \Gamma_2({\gamma'}^{\pm\frac{1}{2}})={\gamma'}^{\pm\frac{1}{2}},
\qquad \Gamma_2(x_i^{\pm}(k))=a^{k}x_i^{\pm}(k)$$
$$\Gamma_2(\omega_i(k))=a^{k}\omega_i(k),\qquad \Gamma_2(\omega'_i(k))=a^{k}\omega'_i(k),
\qquad \Gamma_2({\omega}_i)=\omega_i,\qquad
\Gamma_2({\omega'}_i)=\omega'_i.$$
\end{prop}

\section{Finite-dimensional representations of $U_{r, s}(\widehat{\mathfrak{g}})$}

\noindent {\bf 3.1}\, In this subsection, we study finite-dimensional representation
theory of $U_{r, s}(\widehat{\mathfrak{g}})$ analogue to the
one-parameter situation (\cite{CP1}-\cite{CP3}). For finite dimensional
$\mathfrak g$ the highest weight representations of $U_{r,
s}({\mathfrak{g}})$ have been discussed in \cite{BW2} and \cite{BGH2}.

Let $P$ be the weight lattice of the simple Lie algebra $\mathfrak{g}$,
and $W$ be a representation of $U_{r,\,s}(\mathfrak{g})$. We say
$\lambda\in P$ is a weight of $W$, if the weight space
$$W_\lambda=\{ w\in W|\,\om_i w=\langle \om_\lambda,\, \om_i\rangle w,\,\om'_i w=\langle \om_i,\,
\om_\lambda\rangle^{-1}w ,\, \forall i \}\neq 0.$$
Here we have extended the definition of $\langle  \, \, \, \rangle$ from
$\lambda \in Q$ to $\lambda \in P$ via appropriate
half-integer powers when necessary. A representation $W$ of
$U_{r,\,s}(\mathfrak{g})$ is said to be of type 1, if it is the direct
sum of its weight spaces, that is,
$$W=\bigoplus\limits_{\lambda \in wt(W)} W_\lambda,$$
where $wt(W)$ is the set of weight of $W$.

A non-zero vector $w\in W_\lambda$ is called a highest weight vector
if $e_i \cdot w=0$ for all $i\in I$, and $W$ is a highest weight
representation with highest weight $\lambda$ if
$W=U_{r,\,s}(\mathfrak{g}) w$ for some highest weight vector $w\in
W_\lambda$. Any highest weight representation is of type 1.

We now turn to the representation theory of
$U_{r,\,s}(\hat{\mathfrak{g}})$. A representation $V$ of
$U_{r,\,s}(\hat{\mathfrak{g}})$ is of type 1 if $\gamma^{\frac{1}{2}}$
and $\gamma'^{\frac{1}{2}}$ act as the identity on $V$, and if $V$
is of type 1 as a representation of $U_{r,\,s}(\mathfrak{g})$. A vector
$v\in V$ is a highest weight vector if
$$x_i^+(k)\cdot v=0,\, \om_i(m)\cdot v= \Phi_{i\,m}^{+}v,\,\om'_i(-m)\cdot v= \Phi_{i\,-m}^{-}v,\,
\gamma^{\frac{1}{2}}\cdot v=v,\,\gamma'^{\frac{1}{2}}\cdot v=v,$$
for some complex number $\Phi_{i\,\pm m}^{\pm}$. A type 1
representation $V$ is a highest weight representation if
$V=U_{r,\,s}(\hat{\mathfrak{g}}) v$ for some highest weight vector $v$,
and Let $\epsilon=\pm=\pm 1$, the pair of $(I\times
\mathbb{Z})$-tuples $(\Phi_{i,\,\epsilon m}^{\epsilon})_{i \in I,\,
m\in\mathbb{Z}_{\geq 0}}$ is called the highest weight of $V$.

\smallskip


The following result can be proved similarly as in the one-parameter
case. In fact they are exactly proved as in the classical cases
(cf. \cite{CP1}-\cite{CP3}).

\begin{prop}
If $V$ is a finite dimensional irreducible  representation of
$U_{r,\,s}(\hat{\mathfrak{g}})$, then

 (1)\,$V$ can be obtained from a type 1 representation by
twisting with a product of an automorphism $a_{\sigma}$.

 (2)\, If
$V$ is of type 1, then $V$ is a highest weight module.
\end{prop}

\medskip

 \noindent {\bf 3.2}\, In this subsection. we give the the main results on the finite-dimensional representation
 of $U_{r, s}(\widehat{\mathfrak{g}})$. Though the results are similar to the
 one-parameter case, they also carry some different aspects for the
 two-parameter case.

 If $\lambda \in P^+$, let
$\mathcal{P}^{\lambda}$ be the set of $I$-tuples $(P_i)_{i\in I}$ of
polynomials $P_i\in \mathbb{C}[u]$, with constant term 1, such that
$deg(P_i)=\lambda(i)$ for all $i\in I$. Set
$\mathcal{P}=\bigcup\limits_{\lambda\in P^+}\mathcal{P}^{\lambda}$.

\begin{theo}\label{T:drin1}
Let $\bf{\Phi^\epsilon}=(\Phi_{i,\,\epsilon m}^\epsilon )_{in \in
I,\, m\in\mathbb{Z}}$ be a pair of $(I\times \mathbb{Z})$-tuples of
complex numbers, then the irreducible representation
$V\bf{(\Phi,\,\Psi)}$ of $U_{r,\,s}(\hat{\mathfrak{g}})$ is
finite-dimensional if and only if there exists $\bf{P}=(P_i)_{i\in
I}\in \mathcal{P}$ such that
$$\sum\limits_{m=0}^{\infty}\Phi_{i,\,\epsilon m}^\epsilon z^{\epsilon m}
=r_i^{deg(P_i)}\frac{P_i((rs)^{\frac{1-\epsilon}{2}deg(P_i)}s_iz)}{P_i((rs)^{\frac{1-\epsilon}{2}deg(P_i)}r_iz)}
$$ in the sense that the positive and the negative terms are the Laurent
expansions of the middle term about $0$ and $\infty$, respectively.
\end{theo}

\begin{proof} We will discuss the case of $U_{r,\,s}(\hat{\mathfrak{sl}_2})$
in next section. Assuming the result is true for
$U_{r,\,s}(\hat{\mathfrak{sl}_2})$, we can argue the general case
easily. First of all, the necessary condition (the ``only if'' part)
is trivial. Now suppose that there exists a family of polynomials
$\bf{P}=(P_i)_{i\in I}\in \mathcal{P}$ satisfying the condition in
the theorem. For each $i\in I$, as a module for the $i$th copy
$U_{r,\,s}(\hat{\mathfrak{sl}_2})$, the existence of Drinfeld
polynomials implies that for each $i\in I$ the span of irreducible
$i$th $U_{r,\,s}(\hat{\mathfrak{sl}_{2}})$-module is finite dimensional.
However, as in the case of quantum affine algebras, the union
of these spans are exactly the whole module for
$U_{r,\,s}(\hat{\mathfrak{g}})$, so it is also finite dimensional.
\end{proof}

Assigning to $V$ the n-tuple $P$ defines a bijection between the set
of isomorphism classes of finite dimensional irreducible
representation of $U_{r, s}(\widehat{\mathfrak{g}})$ of type 1 and $P$.
We denote the finite-dimensional irreducible representation of
$U_{r, s}(\widehat{\mathfrak{g}})$ associated to $\bf{P}$ by
$V(\bf{P})$, and we will simply say that $\bf{P}$ its highest weight.

\begin{remark}\,
For each $i\in I$ and $a\in \mathbb{C}^*$, we can define the
irreducible representation $V_{w_i}(a)$ as $V(P_a^{(i)})$, where
$P_a^{(i)}=(P_1,\,\cdots, \, P_n)$ is the n-tuple of polynomials,
such that $P_i(u)=1-au, \, P_j(u)=1,$ for all $j\neq i$. We call
$V_{w_i}(a)$ the ith fundamental representation of $U_{r,
s}(\widehat{\mathfrak{g}})$. As we will see in the following
they constitute the building blocks
for general finite dimensional modules.
\end{remark}

\begin{theo}\label{T:tensor}
Let $\bf{P},\,\bf{Q}\in \mathcal{P}$ be as above, Let $v_{\bf {P}}$
and $v_{\bf {Q}}$ be highest weight vector of $V(\bf{P})$ and
$V(\bf{Q})$, respectively. If $\bf{P}\otimes \bf{Q}$ denotes the
$I-$tuple $(P_iQ_i)_{i\in I}$. Then $V(\bf{P}\otimes \bf{Q})$ is
isomorphic to a quotient of the submodule of $V(\bf{P})\otimes
V(\bf{Q})$ generated by the tensor product of the highest weight
vectors $v_{\bf {P}}$ and $v_{\bf {Q}}$.
\end{theo}

\begin{proof}\, This proof is almost the same as one-parameter case. By abuse of notation, if ${\bf P}=(P_i)_{i\in I}$, ${\bf Q}=(Q_i)_{i\in I}$,
we denote $\bf{P}\otimes \bf{Q}\in \mathcal{P}$ be the $I$-tuple $(P_iQ_i)_{i\in I}$. Then in
$V(\bf{P})\otimes V(\bf{Q})$, for all $i\in I$, $k,\,m\in \mathbb{Z}$, we have
$$x_{i}^+(k)\cdot (v_{\bf {P}}\otimes v_{\bf {Q}})=0,
\quad \omega_i(m)\cdot (v_{\bf {P}}\otimes v_{\bf {Q}})=\Phi_{i\,m}^{+} (v_{\bf {P}}\otimes v_{\bf {Q}}),$$
$$\omega'_i(m)\cdot (v_{\bf {P}}\otimes v_{\bf {Q}})=\Phi_{i\,m}^{-} (v_{\bf {P}}\otimes v_{\bf {Q}}).$$
Thus it is an direct consequence of the above arguments.
\end{proof}

The following result is an immediate consequence of Theorem \ref{T:tensor}.
\begin{coro}
Any finite dimensional irreducible module of $U_{r,
s}(\widehat{\mathfrak{g}})$ of type 1 is isomorphic to an quotient of
the submodule of a tensor product of fundamental representations.
\end{coro}

\noindent {\bf 3.3}\, Similar to the one-parameter case, for any
representation $V$ of $U_{r, s}(\widehat{\mathfrak{g}})$ , we can
decompose $V$ into a direct sum $V=\oplus V_{\gamma_{i,m}^{\pm}}$,
where $$V_{\gamma_{i,m}^{\pm}}=\{x\in V|(\Phi_{i,\pm
m}^{\pm}-{\gamma_{i,m}^{\pm}})^p\cdot x=0, \hbox{for some p},
\forall i, \, m\}.$$

Given a collection $({\gamma_{i,m}^{\pm}})$ of eigenvalues, we form
the generating functions
$$\gamma_{i}^{\pm}(u)=\sum\limits_{m>0}\gamma_{i,\pm m}^{\pm}u^{\pm m}.$$

The following result generalizes the corresponding result for one-parameter cases
\cite{FR} and also generalizes Theorem \ref{T:drin1}.

\begin{prop}\, The generating functions $\gamma_{i}^{\pm}(u)$ of eigenvalues  on any finite
dimensional representation of $U_{r, s}(\widehat{\mathfrak{g}})$ have
the form
$$\gamma_{i}^{\pm}(u)=r_i^{deg R_i-\frac{deg Q_i}{2}}s_i^{\frac{deg Q_i}{2}}\frac{R_i(us_i)Q_i(ur_i)}{R_i(ur_i)Q_i(us_i)}$$
as elements of $\mathbb{C}[[u]]$ and $\mathbb{C}[[u^{-1}]]$,
respectively, where $Q_i(u), R_i(u)$ are polynomials in $u$.
\end{prop}

\begin{proof} The case of $U_{r,\,s}(\hat{\mathfrak{sl}_2})$
will be discussed in next section. For the general case,
one can view the module as a finite dimensional
module for each $U_{r,\,s}^{(i)}(\hat{\mathfrak{sl}_2})$,
the $i$th copy of the subalgebra of $U_{r,\,s}^{(i)}(\hat{\mathfrak{g}})$.
One sees that the generating function $\gamma_{i}^{\pm}(u)$
must be of the stated form.
\end{proof}

\section{The case of  $U_{r, s}(\widehat{\mathfrak{sl}_2})$}
\smallskip

\noindent {\bf 4.1}\,  We recall the evaluation map of the
two-parameter quantum
 affine algebra $U_{r, s}(\widehat{\mathfrak{sl}_2})$ in this subsection and give necessary computations
 (cf. \cite{ZP}). As in the one-parameter case, $U_{r, s}(\widehat{\mathfrak{sl}_2})$ plays
 a similar role for the general theory. This is in alignment with the general principle
 that $sl(2)$ effectively controls the general Lie theory. For this reason, we will give
 detailed computation for the case of $sl(2)$. Our computation reconfirms
 this principle for the general two-parameter quantum affine algebras.

\begin{prop} For any $a\in \mathbb{C}^*$, there exists an algebra morphism $\mathrm{ev}_a$ from
$U_{r,s}(\widehat{\mathfrak{sl}_2})$  to $U_{r,s}(\mathfrak{sl}_2)$ defined
as follows:
$$\mathrm{ev}_a(e_0)=r^{-1}sa\, f,\quad \mathrm{ev}_a(f_0)=rs^{-1}a^{-1}\,e,\quad \mathrm{ev}_a(e_1)=e,\quad \mathrm{ev}_a(f_1)=f,$$
$$\mathrm{ev}_a(\om_0)=\om',\quad \mathrm{ev}_a(\om_1)=\om,\quad \mathrm{ev}_a(\om'_0)=\om,\quad \mathrm{ev}_a(\om'_1)=\om'.$$
\end{prop}
\smallskip

Using the Drinfeld isomorphism theorem, we can lift the
evaluation morphism  $\mathrm{ev}_a$ to its Drinfeld realization
$\mathcal{U}_{r,s}(\widehat{\mathfrak{sl}_2})$, we also denote it by
$\mathrm{ev}_a$ without ambiguity.
\begin{prop}\label{EV1}\, There exists an algebra morphism $\mathrm{ev}_a$ from
$\mathcal{U}_{r,s}(\widehat{\mathfrak{sl}_2})$  to
$U_{r,s}(\mathfrak{sl}_2)$ defined as follows:
$$\mathrm{ev}_a(\gamma)=1=\mathrm{ev}_a(\gamma'),\,\quad \mathrm{ev}_a(\om)=\om,
\quad \mathrm{ev}_a(\om')=\om',$$
$$\mathrm{ev}_a(x^+(k))=r^{-k}s^ka^k\om'^{-k} e,\quad \mathrm{ev}_a(x^-(k))=r^{-k}s^ka^kf\om^{k} .$$
\end{prop}
\smallskip

\smallskip

Similar to the one-parameter case, we can define the evaluation
representation of $\mathcal{U}_{r,s}(\widehat{\mathfrak{sl}_2})$ as
follows:

\begin{defi}
For $a\in \mathbb{C}*$ and $n\in \mathbb{N}$, we call these $V_n(a)$
evaluation representations of quantum affine algebra
$\mathcal{U}_{r,s}(\widehat{\mathfrak{sl}_2})$.
\end{defi}
\smallskip

We recall the representation of $U_{r,s}({\mathfrak{sl}_2})$, see
\cite{BW2} and \cite{BGH2} for more detail.  For the representation
$V_n$ of $U_{r,s}({\mathfrak{sl}_2})$ of dimension n+1, there exists a
basis $v_0,\, v_1,\, \cdots,\, v_n$ of $V_n$ such that $$\om \cdot
v_i=r^n(rs^{-1})^{-i}v_i,\qquad \om' \cdot
v_i=s^n(rs^{-1})^{i}v_i,$$
$$ e \cdot
v_i=[n+1-i]v_{i-1},\qquad f \cdot v_i=[i+1]v_{i+1}.$$
\smallskip

\begin{prop}
As a vector space, $V_n(a)$ is equal to $V_n$, and the action of
$\mathcal{U}_{r,s}(\widehat{\mathfrak{sl}_2})$  is given by:
$$x^+(k)\cdot v_i=a^ks^{-nk}(rs^{-1})^{-ki}[n+1-i]\, v_{i-1},$$
$$x^-(k)\cdot v_i=a^kr^{nk}(rs^{-1})^{-k(i+1)}[i+1]\, v_{i+1}.$$
\end{prop}
\medskip

A vector $v\in V$ is a highest weight vector if for $l\in
\mathbb{Z},\, k\in \mathbb{Z}_{\geqslant 0}$, we have
$$x^+(l)\cdot v=0,\quad a(k)\cdot v=d_k^+ v,\quad a(-k)\cdot v=d_{-k}^- v, \quad \gamma \cdot v=\gamma'\cdot v=1$$
for some complex numbers $\Phi_k^+$ and $\Phi_{-k}^-$. Note that
$\Phi_0^+\Phi_0^-$ is the powers of (rs). We will show that
evaluation representations $V_n(a)$ for $a\in \mathbb{C}^*$ are
highest weight representations.

It is easy to see from the above proposition that $v_0$ is
annihilated by all $x^+(k)$. The actions of $\om_0$ and $\om'_0$ can
be easily computed
$$\om_0\cdot v_0= r^{n}v_0,\qquad \om'_0\cdot v_0= s^{n}v_0.$$

In general, we have

\begin{prop}

All finite dimensional irreducible representations of
$U_{r,s}({\mathfrak{sl}_2})$ are highest weight representations.
\end{prop}

\medskip

\begin{theo}
(1)\, Let $V$ be a finite dimensional highest weight representation
of $U_{r,s}({\mathfrak{sl}_2})$, and as an evaluation module for
 $U_{r,s}({\widehat{\mathfrak{sl}}_2})$ there exists a polynomial $P(z)\in \mathbb{C}[z]$ such that
$P(0)\neq 0$ and
$$\sum\limits_{k=0}^{\infty}\Phi_k^+z^k=r^{degP}\frac{P(sz)}{P(rz)},$$
$$\sum\limits_{k=0}^{\infty}\Phi_{-k}^-z^{-k}=r^{degP}\frac{Q(sz)}{Q(rz)},$$
where $Q(z)=P((rs)^{degP}z)$ .

(2)\,For any series of complex number
$\underline{\Phi}=(\Phi_k^+,\,\Phi_{-k}^-)_{k\in \mathbb{N}}$ such
that $\Phi_o^+\Phi_0^-=(rs)^n$ for some integer $n$, there exists a
finite dimensional irreducible highest weight module
$V(\underline{\Phi})$.
\end{theo}
The polynomial $P(z)$ in the theorem is called Drinfel'd polynomial.
And the following proposition is an example the above theorem.

\begin{prop}
For the evaluation representation $V_n(a)$ of
$U_{r,s}({\widehat{\mathfrak{sl}}_2})$, we obtain the Drinfel'd
polynomial as follows.
$$P(z)=\sum\limits_{k=1}^{n}(1-ar^{-k-1}s^{k-n}z).$$
\end{prop}

\begin{proof}\, We first have
$$\Phi_k^+=(r-s)(ar^{-1}s^{1-n})^k[n]$$ and
$$\Phi_{-k}^-=-(r-s)(a^{-1}r^{1-n}s^{-1})^k[n],$$ where
$[n]=\frac{r^n-s^n}{r-s}$.

It is easy to see that
\begin{eqnarray*}
\sum\limits_{k=0}^{\infty}\Phi_k^+z^k&=&r^{n}+\sum\limits_{k=1}^{n}(r^n-s^n)\frac{ar^{-1}s^{1-n}z}{1-ar^{-1}s^{1-n}z}\\
 &=&r^n\frac{1-a(r^{-1}s)r^{-n} z}{1-a(r^{-1}s)s^{-n}z}\\
 &=&r^n(\frac{1-a(r^{-1}s)^2s^{-n} z}{1-a(r^{-1}s)s^{-n}z})(\frac{1-a(r^{-1}s)^3s^{-n}
 z}{1-a(r^{-1}s)^2s^{-n}z})\cdots (\frac{1-a(r^{-1}s)^{n+1}s^{-n}
 z}{1-a(r^{-1}s)^ns^{-n}z})\\
 &=&r^{n}\frac{P(sz)}{P(rz)}
\end{eqnarray*}
where $P(z)=\sum\limits_{k=1}^{n}(1-a(r^{-1}s)^kr^{-1}s^{-n}z).$

Similarly, we have
$$\sum\limits_{k=0}^{\infty}\Phi_{-k}^-z^{-k}=r^{degP}\frac{Q(sz)}{Q(rz)},$$
where
$Q(z)=\sum\limits_{k=1}^{n}(1-a(r^{-1}s)^kr^{n-1}z)=P((rs)^nz).$

\end{proof}
The following proposition gives the Drinfeld polynomials for all
evaluation representations of two-parameter quantum affine algebra
$\mathcal{U}_{r,s}(\widehat{\mathfrak{sl}_2})$.

\begin{prop} For the evaluation representation
$W_n(a)=V_n(rs^{-1}a)$ of two-parameter quantum affine algebra
$\mathcal{U}_{r,s}(\widehat{\mathfrak{sl}_2})$, we can obtain the
eigenvalues $\Phi_{k,j}$ in the above proposition 3.6. as follows:
$$\sum\limits_{m>0}\Phi_{i,\pm m}^{\pm}u^{\pm m}=r^{deg R_i-\frac{deg Q_i}{2}}s^{\frac{deg Q_i}{2}}\frac{R_i(us)Q_i(ur)}{R_i(ur)Q_i(us)}$$
where
$$Q_i(z)=\prod\limits_{j=1}^i(1-ar^{-j}s^{j-n-1}u)(1-ar^{1-j}s^{j-n-2}u),$$
$$R_i(z)=\prod\limits_{j=1}^n(1-ar^{-j}s^{j-n-1}u).$$
\end{prop}

\begin{proof}\, It is easy to see that $W_n(a)$ admits a
linear space decomposed as ${U}_{r,s}({\mathfrak{sl}_2})-$module as
follows:
$$W_n(a)=\mathbb{C}v_0\oplus \mathbb{C}v_1 \oplus\cdots \oplus\mathbb{C}v_n.$$

It follows from proposition 4.4 that
$$x^+(k)\cdot v_i=a^ks^{-nk}(rs^{-1})^{-ki}[n+1-i]\, v_{i-1},$$
$$x^-(k)\cdot v_i=a^kr^{nk}(rs^{-1})^{-k(i+1)}[i+1]\, v_{i+1}.$$
So by direct calculation, we get,
\begin{eqnarray*}
&& \omega(k)\cdot
v_i\\\
&=&(r-s)a^ks^{-nk}(rs^{-1})^{-ki}(rs^{-1})^k\big((rs^{-1})^{-k}[i+1][n-i]-[n+1-i][i]\big)\cdot
v_i
\end{eqnarray*}
Then we obtain
\begin{eqnarray*}
&&\sum\limits_{k=0}^{\infty}\Phi_{i,k}^+ u^k\\
&=&r^n(rs^{-1})^{-i}+
\sum\limits_{k=1}^{\infty}(r-s)a^ks^{-nk}(rs^{-1})^{-ki}(rs^{-1})^k\\
&&\Big((rs^{-1})^{-k}[i+1][n-i]-[n+1-i][i]\Big)\\
&=&r^n(rs^{-1})^{-i}+(r-s)\Big(\frac{ar^{-i}s^{i-n}u}{1-ar^{-i}s^{i-n}u}[i+1][n-i]\\
&&\hskip3.5cm-\frac{ar^{1-i}s^{i-n-1}u}{1-ar^{1-i}s^{i-n-1}u}[n+1-i][i]\Big)\\
&=&r^{n-i}s^{i}\,\frac{(1-ars^{-n-1}u)(1-ar^{-n}u)}{(1-ar^{-i}s^{i-n}u)(1-ar^{1-i}s^{i-n-1}u)}\\
&=&r^{n-i}s^{i}\,\frac{R_i(us)Q_i(ur)}{R_i(ur)Q_i(us)}
\end{eqnarray*}
where$$Q_i(z)=\prod\limits_{j=1}^i(1-ar^{-j}s^{j-n-1}u)(1-ar^{1-j}s^{j-n-2}u),$$
$$R_i(z)=\prod\limits_{j=1}^n(1-ar^{-j}s^{j-n-1}u).$$
The another relation is similar. Thus we have completed the proof.
\end{proof}

\section{Specializations}

\noindent {\bf 5.1}\,
By our previous analysis of Drinfeld polynomials,
the theory of finite dimensional representations of $U_{r, s}(\widehat{\mathfrak{sl}_2})$
is quite similar to the classical case in the generic case when the parameters $r$ and
$s$ are independent.
When we consider specializations of the two parameters $r$ and
$s$, there are some special phenomena.

(I) If we specialize  $s$ to $r^{-1}$, then the quantum Cartan
matrix of $U_{r,s}(\widehat{\mathfrak{sl}_2})$ becomes
$$\left(\begin{array}{cc}
r^2& r^{-2} \\
r^{-2} & r^{2}
\end{array}\right),$$ which is the same to that of the classical case.

(II) If we specialize  $s$ to $r$, then we have $[n]=nr^{n-1}$, and
the quantum Cartan matrix of $U_{r,s}(\widehat{\mathfrak{sl}_2})$
becomes
$$\left(\begin{array}{cc}
1& 1 \\
1 & 1
\end{array}\right),$$ which implies the group-like elements $\omega$
and $\omega'$ are in the center of $U_{r,s}(\widehat{\mathfrak{sl}_2})$.
On the other hand, from the actions of generators $\omega, \omega'$:
$$\omega\cdot v_i=r^n v_{i};\quad \omega'\cdot v_i=r^n v_{i};\quad\omega_o\cdot v_i=r^n v_{i};\quad \omega'_0\cdot v_i=r^n v_{i},$$
they have the same eigenvalue $r^n$, while the finite dimensional
representation $V_n$ is still irreducible.

(III) If we specialize $r$ to $s^k$, where $k\in \mathbb{Z}/\{1\}$,
or $r$ and $s$ are independent,  then we just let $q^2=rs^{-1}$,
quantum Cartan matrix of $U_{r,s}(\widehat{\mathfrak{sl}_2})$ becomes
$$\left(\begin{array}{cc}
q^2& q^{-2} \\
q^{-2} & q^{2}
\end{array}\right).$$ Thus the finite dimensional representation theory
is similar to that of one-parameter case.

 \vskip30pt \centerline{\bf ACKNOWLEDGMENT}

\bigskip
Jing thanks the support of
Simons Foundation grant 198129, NSFC grant 11271138 and NSF grants
1014554 and 1137837.  H. Zhang would
like to thank the support of NSFC grant 11371238 and 11101258.

\bigskip

\bibliographystyle{amsalpha}

\end{document}